\documentclass[11pt,reqno]{amsart}

\usepackage{amsmath, amsfonts, amsthm, amssymb, color}

\newtheorem{Theorem}{Theorem}[section]
\newtheorem{Lemma}{Lemma}[section]
\newtheorem{Proposition}{Proposition}[section]

\theoremstyle{definition}
\newtheorem{Definition}{Definition}[section]

\theoremstyle{remark}
\newtheorem{Remark}{Remark}[section]

\numberwithin{equation}{section}

\renewcommand{\u}{{\bf u}}

\newcommand{\R}{{\mathbb R}}
\newcommand{\Dv}{{\rm div}}

\newcommand{\x}{{\bf x}}

\def\f{\frac}

\def\hf1{^\f{1}{1-\xi^2}}

\def\be{\begin{equation}}
\def\en{\end{equation}}
\def\bs{\begin{split}}
\def\es{\end{split}}

\author{Cheng Yu}
\address{Department of Mathematics,  The University of Texas,
                           Austin, Texas 78712.}
\email{yucheng@math.utexas.edu}

\title
[]
{ The weak solution to a  Boltzmann type equation and its energy conservation}

\subjclass[2000]{}

\date{\today}

\begin{document}
\begin{abstract}
In this paper, we study the initial value problem of  a  Boltzmann type equation with a nonlinear
degenerate damping. We prove the existence of global weak solutions with large initial data,  in three
dimensional space. We rely on a variant version of the Gronwall inequality and $L^p$ regularity of average
 velocities to derive the compactness of solutions to a suitable approximation. This allows us to recover a weak
 solution by passing to the limits. After the existence result, we also prove energy conservation for the weak solution under some certain condition.
\end{abstract}

\maketitle

\section{Introduction}

Kinetic approach plays a critical role in many variant fields of mathematical physics and applied sciences, from micro- and nano-physics to continuum mechanics, and from social science to biological science. It is an important tool
in the modeling and simulation of phenomena across length and time scales, from the atomistic to the continuum. Thus,
it has diverse applications in gas dynamics, engineering, medicine,
geophysics and astrophysics, and has attracted numerous mathematical interests in modeling and analysis, see the references \cite{AOB,BDM,BDGM,GM,H,LV,MV,OR,ODA,YU,W}.
Naturally, one of the fundamental problems is to study the existence of weak solution and its energy conservation.
\vskip0.3cm

 In this paper we consider the following nonlinear partial differential equation \cite{AOB,BDM,LV,ODA}:
\begin{equation}
\label{kinetic equation-1}
f_t+\xi\cdot\nabla_x f+\Dv_{\xi}(fF)=\mathfrak{Q}(f),
\end{equation}
where $f(t,x,\xi)$ is the density function for individuals at time $t\in \R^+$, physical position of an individual $x\in\R^3$, with the velocity  $\xi\in\R^3$.
Thus, $f(t,x,\xi)\,dx\,d\xi$ is the number density of individuals at position between $x$ and $x+dx$ and with velocity between $\xi$ and $\xi+d\xi$. The function
$$\int_{\R^3}f(t,x,\xi)\,d\xi$$
denotes the number density of individuals at the physical position $x$ at time $t$.
The evolution of density distribution $f(t,x,\xi)$ is described by equation \eqref{kinetic equation-1}.
Note that, $F$ is the external force acting on the individuals, and the operator $\mathfrak{Q}(f)$ denotes the rate of change of $f$ due to reaction, random choice of velocity, etc.
We assume that the external force $F\equiv 0$ in this paper.
\vskip0.3cm

Let us to give some backgrounds on the operator $\mathfrak{Q}(f)$. Assume that two different processes contribute to $\mathfrak{Q}(f)$, that is $$\mathfrak{Q}=\mathfrak{Q}_1+\mathfrak{Q}_2,$$
where $\mathfrak{Q}_1$ denotes a birth-death process, $\mathfrak{Q}_2$ stands for a process that generates random velocity changes.
The birth-death process can be described as follows
$$\mathfrak{Q}_1(f)=-\mu(\mathfrak{n})f,$$
where $\mu=\mu(\mathfrak{n})\geq 0$ is the birth-death rate which only depends on the number density of individuals $\mathfrak{n}(t,x)$. For more details on this operator, we refer the reader to \cite{GM,ODA}.
The number density of individuals is given by
$$\mathfrak{n}(t,x)=\int_{\R^3}f(t,x,\xi)\,d\xi.$$ It is also called the particles density (or zeroth moment) in the kinetic theory.
The stochastic process is modeled by the following operator(see \cite{ODA}),
\begin{equation}
\label{operator}
\mathfrak{Q}_2(f)=-\lambda f(t,x,\xi)+\lambda\int_{\R^3} T(\xi,\xi')f(t,x,\xi')\,d\xi',
\end{equation}
where $\lambda>0$ is the break-up frequency. The kernel $T(\xi,\xi')$ is the probability of a change with respect to velocity from $\xi'$ to $\xi$, and $\xi$ is the velocity of individuals before the collision while $\xi'$ is the velocity immediately after the collision. Thus, equation \eqref{kinetic equation-1} becomes to the following
\begin{equation*}
f_t+\xi\cdot\nabla_x f=-\mu(\mathfrak{n}) f-\lambda f+\lambda\int_{\R^3}T(\xi,\xi')f(t,x,\,\xi')\,d\xi^{'}.
\end{equation*}

\vskip0.3cm

Furthermore, from the conservation of kinetic energy it follows,
$$|\xi|^2=|\xi'|^2,$$
from which, we deduce
\begin{equation}
 \label{same speed}|\xi|=|\xi'|.
 \end{equation}
Thus, of particular interest in this paper is the case in which the speed does not change with reorientation.
 Given that a reorientation occurs,
  the probability function $T(\xi,\xi')$ is a non-negative function and after normalization we may have \begin{equation}
 \label{probability function}
 \int_{\R^3} T(\xi,\xi')\,d\xi=1.\end{equation}
 Next, following the work of \cite{LV}, we assume that $T(\xi,\xi')$ satisfies a self-similarity property, namely
 \begin{equation}
 \label{self-similarity}
 T(\xi,\xi')=H(|\xi'|)T(\frac{\xi}{|\xi'|},\frac{\xi'}{|\xi'|}),\quad\text{ for some function } H(\cdot).
 \end{equation}
In this current paper, we assume that  \eqref{same speed}, \eqref{probability function} and \eqref{self-similarity} hold.

 \vskip0.3cm

The objective of our current work is to investigate the issue of existence
of global weak solutions  to
\begin{equation}
\label{kinetic equation}
f_t+\xi\cdot\nabla_x f=-\mu(\mathfrak{n}) f-\lambda f+\lambda\int_{\R^3}T(\xi,\xi')f(t,x,\,\xi')\,d\xi^{'}
\end{equation} with the following initial data:
\begin{equation}
\label{initial data}
f(0,x,\xi)=f^0(x,\xi),
\end{equation}
where $\mathfrak{n}(t,x)=\int_{\R^3}f\,d\xi,$  and $(t,x,\xi)\in \R^{+}\times \R^3\times\R^3.$ The first goal is to prove the global existence of weak solutions with large initial data. After the existence result, we shall prove the energy conservation for such a weak solution. Note that, the energy conservation is a fundamental problem in the physical theory and the mathematical study of kinetic equation.

\vskip0.3cm
To the best of our knowledge, the first existence result related to \eqref{kinetic equation} goes back to
the work of
Leger-Vasseur \cite{LV} where they established the existence of solutions to the coupled system by Navier-Stokes and \eqref{kinetic equation} with $\mu(\mathfrak{n})=0.$  In \cite{LV} they constructed
the weak solutions to the associated kinetic equation for $(x,\xi)$
in bounded domains. Roughly speaking, they
build a sequence of nonnegative solutions $\{f_n\}_{n=1}^{\infty}$ to a suitable approximation. Because the energy inequality does not hold at the approximation level, the uniform estimates cannot be derived directly. They proved this sequence is an increasing sequence in $n$ by induction,  and applied the Monotone Convergence Theorem to $\{f_n\}_{n=1}^{\infty}$ to derive its strong convergence. Note, the energy inequality does not hold at the approximated iteration. Combined the increasing of $\{f_n\}_{n=1}^{\infty}$, the bad terms in energy inequality can be bounded by means of $L^p$ estimates in bounded domain.  We point out that the Monotone Convergence Theorem is a crucial tool to handle the kinetic equation in \cite{LV}.
Our main goal of this paper is to  extend the existence result of kinetic equation in \cite{LV} to \eqref{kinetic equation} with $\mu(\mathfrak{n})$ in the setting $\R^+\times\R^3\times\R^3$.
The nonlinear term $\mu(\mathfrak{n})f$
leads to the loss of the monotonicity for the sequences of solutions to the suitable approximation, the arguments of  Leger-Vasseur \cite{LV} cannot apply for the convergence.
For this reason we need to develop new argument  to guarantee the convergence of solutions to approximation.

\vskip0.3cm

Note that, the smooth solution of \eqref{kinetic equation}-\eqref{initial data}, satisfies the energy equality. In particular,  we have energy inequality
\begin{equation}
\label{energy inequality for definition}
\begin{split}
\int_{\R^3\times\R^3}(1+|\xi|^m)f\,d\xi\,d x&+ \int_0^T\int_{\R^3\times\R^3}\mu(\mathfrak{n})(1+|\xi|^m)f\,d\xi\,dx\,dt
\\
&\leq \int_{\R^3\times\R^3}(1+|\xi|^m)f^0\,d\xi\,d x,
\end{split}
\end{equation}
for any $T>0$. Energy inequality \eqref{energy inequality for definition} will derive in Section 2. Thus, it is natural to ask the initial data satisfies the following ones
\begin{equation}
\label{condition on initial data}
 \int_{\R^3\times\R^3}(1+|\xi|^m)f^0\,d\xi\,d x<+\infty.
\end{equation}
In fact, we will prove the stronger version of \eqref{energy inequality for definition} in Section 5, for a weak solution under some certain condition,
\begin{equation*}
\begin{split}
\int_{\R^3\times\R^3}(1+|\xi|^m)f\,d\xi\,d x&+ \int_0^T\int_{\R^3\times\R^3}\mu(\mathfrak{n})(1+|\xi|^m)f\,d\xi\,dx\,dt
\\
&= \int_{\R^3\times\R^3}(1+|\xi|^m)f^0\,d\xi\,d x,
\end{split}
\end{equation*}
for any $t\in[0,T].$ This implies that such a weak solution preserves the energy conservation.

 \vskip0.3cm
 Based on energy inequality \eqref{energy inequality for definition} and related estimates, we give the definition of weak solutions in the following sense.
\begin{Definition} \label{D1} The function
 $f$ is a global weak solution to the initial value problem \eqref{kinetic equation}-\eqref{initial data} if,  for any $T>0$, the following properties hold,\\
\begin{itemize}
\item The function $f(t,x,\xi)$ has the following regularities
  \begin{equation*}
\begin{cases}
&f(t,x,\xi)\geq 0\quad\text{ for any } (t,x,\xi)\in (0,T)\times\R^3\times\R^3;
\\&f\in C(0,T;L^1(\R^3\times\R^3)\cap L^{\infty}(0,T; L^1\cap L^{\infty}(\R^3\times\R^3)),
\\&|\xi|^mf \in L^{\infty}(0,T;L^1(\R^3\times\R^3))\quad\text{ for any } m\in [0, m_0] \text{ for some } m_0\geq 3;
\end{cases}
\end{equation*}
\item The initial value problem \eqref{kinetic equation}-\eqref{initial data} holds in the sense of distribution, that is, for any test function $\varphi\in C^{\infty}(\R^3\times\R^3\times[0,T])$, the following weak formulation holds
\begin{equation}
\label{weak formulation in definition}
\begin{split}
&\int_{\R^3\times\R^3}f^0\varphi(0,x,\xi)\,d\xi\,d x+\int_0^T\int_{\R^3\times\R^3}\left(f\varphi_t+\xi f\cdot\nabla_x\varphi\right)\,d\xi\,d x\,dt+
\\&\int_0^T\int_{\R^3\times\R^3}\mu(\mathfrak{n})f\varphi\,d\xi\,d x\,dt=-\lambda\int_0^T\int_{\R^3\times\R^3}f\varphi\,d\xi\,d x\,dt
\\&+\lambda\int_0^T\int_{\R^3\times\R^3\times\R^3}T(\xi,\xi')f(t,x,\xi')\varphi\,d\xi'\,d\xi\,d x\,dt.
\end{split}
\end{equation}
\item Energy inequality holds for almost every where $t>0$:
\begin{equation*}
\begin{split}
\int_{\R^3\times\R^3}(1+|\xi|^m)f\,d\xi\,d x\leq \int_{\R^3\times\R^3}(1+|\xi|^m)f^0\,d\xi\,d x.
\end{split}
\end{equation*}
 \end{itemize}
 \end{Definition}

As our main result, we proved the following theorem on the existence of global weak solutions to the initial value problem \eqref{kinetic equation}-\eqref{initial data}.
\begin{Theorem}
\label{main result}
If $f^0\in L^{\infty}(\R^3\times\R^3)\cap L^{1}(\R^3\times\R^3)$, and $f^0(1+|\xi|^{m})\in L^1(\R^3\times\R^3)$ for any $ m\in[0,m_0]$ for some $m_0\geq 3$; the probability function $T(\xi,\xi')$ satisfies \eqref{probability function} and \eqref{self-similarity};
the function $\mu(\cdot)$ is a Lipschitz continues function;
 then there exists a global weak solution to  the initial value problem \eqref{kinetic equation}-\eqref{initial data}.
\end{Theorem}

\vskip0.3cm

We first follow the approximation introduced in \cite{LV}, which is equation \eqref{regularized equation}. As the same in \cite{LV}, the energy inequality does not hold at the approximation level. However, the sequence of solutions to \eqref{regularized equation} is no longer increasing in $n$ due to a nonlinear term $\mu(\mathfrak{n}_k)f_k$. It is necessary to develop new arguments to facilitate the compactness of solutions to approximation. The crucial step here is to show the strong convergence of the zeroth moment and the first moment in $L^p$ space. By means of the $L^p$ regularity of average velocities, the strong convergence can be obtained
if they are bounded uniformly  in $L^p$ space.
In particular,   we shall rely on  a variant of the Gronwall inequality (see Lemma \ref{Gronwall inequality}) to show that, for any $0\leq t\leq T,$ there exists a constant $K>0$, such that $$\int_{\R^3}|\xi|^m f_k\,d\xi\leq K e^{KT},$$
   for all $k>0.$
 It allows us to control the zeroth moment and first moment $$\int_{\R^3} \xi f_{k}\,d\xi\quad\text{ and }\int_{\R^3} f_k\,d\xi$$
in $L^p$ space for some $p>1$.  In order to assure the convergence $\mu(\mathfrak{n}_{k-1})\mathfrak{n}_k$,  we need to prove the strong convergence of $\mathfrak{n}_k$ in $L^p$ space. This can be done by the $L^p$ regularity of average velocities. Thus
\begin{equation*}
\label{convergence of Q1}
\int_{\R^3}\mu(\mathfrak{n}_{k-1})f_k\,d\xi\to \int_{\R^3}\mu(\mathfrak{n})f\,d\xi
\end{equation*} in $L^{\infty}(0,T;L^1_{loc}(\R^3))$
as $k\to\infty$. With a suitable approximation and the weak stability, the existence of weak solutions can be done.
\vskip0.3cm

Formally, the classical solutions of the physical PDEs always keep the energy conservation, but not a weak solution, at least for fluid equations. Naturally, the question is how badly behaved the solution can be
 in order that a weak solution still can preserve the energy. In fluid equations, there are a lot of study related to this question. For example, Serrin \cite{Serrin} showed if a weak solution $\u$ to incompressible Navier-Stokes equations with additional condition,
then a weak solution $\u$ holds the energy equality for any $0\leq t\leq T.$   Shinbrot \cite{Shinbrot} proved the same conclusion under some different conditions. For the Euler's equation, E-Constantin-Titi \cite{CET} proved the energy equality for a weak solution, which is the answer to the first part of Onsager's conjecture \cite{O}.  As before that we mentioned, the energy conservation is a fundamental problem in physics and related mathematical analysis. Thus, we are interested in studying this question for a weak solution to \eqref{kinetic equation}.
Our second main result is on the energy conservation for a weak solution constructed in Theorem \ref{main result}.

\begin{Theorem}
\label{energy conservation}
Let $f$ be a weak solution constructed in Theorem \ref{main result}, in addition,
\begin{equation}
\label{addition condition on m order}\int_{\R^3\times\R^3}|\xi|^{3+m}f\,d\xi\,d x\leq K<\infty,
\end{equation}
for any $t\in[0,T],$
 then it preserves the energy conservation, that is, $f$ satisfies
\begin{equation}
\label{energy conservation}
\begin{split}
\int_{\R^3\times\R^3}(1+|\xi|^m)f\,d\xi\,d x&+ \int_0^t\int_{\R^3\times\R^3}\mu(\mathfrak{n})(1+|\xi|^m)f\,d\xi\,dx\,dt
\\
&= \int_{\R^3\times\R^3}(1+|\xi|^m)f^0\,d\xi\,d x,
\end{split}
\end{equation}
for any $t\in[0,T].$
\end{Theorem}

\begin{Remark}
The condition \eqref{addition condition on m order} can ensure the first moment $\mathfrak{n}$ is bounded in $L^{\infty}(0,T;L^{\frac{6+m}{3}}(\R^3))$ and ensure the $m$ moments $$\int_{\R^3}|\xi|^mf\,d\xi$$ is bounded in
$L^{\infty}(0,T;L^{\frac{6+m}{3+m}}(\R^3))$. Thus, we find that $$\int_{\R^3}\mu(\mathfrak{n})f(1+|\xi|^m)\,d\xi$$ is bounded in $L^{\infty}(0,T;L^1(\R^3)).$ This is a crucial estimate in showing \eqref{energy conservation}.
\end{Remark}

The rest of the paper is organized as follows. In Section 2, we derive the energy inequality and a crucial estimate on the probability function $T(\xi,\xi').$ In Section 3, we construct a sequence of smooth solutions to an approximation with some uniformly estimates. In Section 4, we show the weak stability of the sequence of solutions and study the limiting process to obtain the existence result of global weak solutions. In Section 5, we will show the energy equality \eqref{energy conservation}.


\bigskip\bigskip

\section{A priori estimates}
In this section, we aim at deriving a priori estimates of  \eqref{kinetic equation}-\eqref{initial data}, which will help us to get the weak stability of the solutions.
Firstly, we derive the energy inequality for any smooth solutions of \eqref{kinetic equation}-\eqref{initial data}, which is the estimate in the following Lemma.
\begin{Lemma}
For any smooth solutions of \eqref{kinetic equation}-\eqref{initial data}, they satisfy the following energy inequality
\begin{equation*}
\begin{split}
\int_{\R^3\times\R^3}(1+|\xi|^m)f\,d\xi\,d x&+ \int_0^T\int_{\R^3\times\R^3}\mu(\mathfrak{n})(1+|\xi|^m)f\,d\xi\,dx\,dt
\\
&\leq \int_{\R^3\times\R^3}(1+|\xi|^m)f^0\,d\xi\,d x,
\end{split}
\end{equation*}
for any $T>0$.
\end{Lemma}
\begin{proof} For any smooth solutions of \eqref{kinetic equation}-\eqref{initial data}, multiplying $(1+|\xi|^m)$ on both sides of \eqref{kinetic equation}, one obtains that
\begin{equation}
\begin{split}
\label{energy law-1}
\frac{d}{dt}\int_{\R^3\times\R^3}(1+|\xi|^m)f\,d\xi\,d x&+ \int_{\R^3\times\R^3}\mu(\mathfrak{n})(1+|\xi|^m)f\,d\xi\,dx
\\&=\int_{\R^3\times\R^3}\mathfrak{Q}_2(f)(1+|\xi|^m)\,d\xi\,dx.
\end{split}
\end{equation}
We calculate the right side term of \eqref{energy law-1} as follows:
\begin{equation}
\begin{split}
\label{operator 2 for energy 1}
\int_{\R^3\times\R^3}\mathfrak{Q}_2(f)&(1+|\xi|^m)\,d\xi\,dx
=-\lambda\int_{\R^3\times\R^3}(1+|\xi|^m)f\,d\xi\,dx
\\&+\lambda\int_{\R^3\times\R^3}\int_{\R^3}T(\xi,\xi')f(t,x,\xi')\,d\xi'(1+|\xi|^m)\,d\xi\,dx.
\end{split}
\end{equation}
 Thanks to Fubini's theorem, \eqref{same speed} and \eqref{probability function},
the second term of right side on  \eqref{operator 2 for energy 1} gives us
\begin{equation}
\label{second term right side}
\begin{split}&
\lambda\int_{\R^3\times\R^3}\int_{\R^3}T(\xi,\xi')f(t,x,\xi')\,d\xi'(1+|\xi|^m)\,d\xi\,dx
\\&=
\lambda\int_{\R^3\times\R^3\times\R^3}T(\xi,\xi')f(t,x,\xi')(1+|\xi'|^m)\,d\xi'\,d\xi\,d x
\\&=\lambda\int_{\R^3\times\R^3}(1+|\xi|^m)f\,d\xi\,dx.
\end{split}
\end{equation}
This, with \eqref{operator 2 for energy 1}, implies
\begin{equation}
\begin{split}
\label{operator 2 for energy}
\int_{\R^3\times\R^3}\mathfrak{Q}_2(f)&(1+|\xi|^m)\,d\xi\,dx
=0.
\end{split}
\end{equation}
 Combining  \eqref{energy law-1} and \eqref{operator 2 for energy}, one obtains
\begin{equation*}
\begin{split}
\frac{d}{dt}\int_{\R^3\times\R^3}(1+|\xi|^m)f\,d\xi\,d x&+ \int_{\R^3\times\R^3}\mu(\mathfrak{n})(1+|\xi|^m)f\,d\xi\,dx
=0,
\end{split}
\end{equation*}
which yields
\begin{equation}
\label{energy law}
\begin{split}
\int_{\R^3\times\R^3}(1+|\xi|^m)f\,d\xi\,d x&+ \int_0^T\int_{\R^3\times\R^3}\mu(\mathfrak{n})(1+|\xi|^m)f\,d\xi\,dx\,dt
\\
&\leq \int_{\R^3\times\R^3}(1+|\xi|^m)f^0\,d\xi\,d x,
\end{split}
\end{equation}
for any $T>0$, at least for the smooth solutions. Thus, it is natural to ask the initial data satisfies the following ones
\begin{equation}
\label{condition on initial data}
 \int_{\R^3\times\R^3}(1+|\xi|^m)f^0\,d\xi\,d x<+\infty.
\end{equation}

\end{proof}

To develop further estimates, we will rely on the following lemma.
 Under the assumption \eqref{self-similarity}, we can show the following Lemma on the probability function $T(\xi,\xi')$.
 \begin{Lemma}
 Assume that $T(\xi,\xi')$ satisfies \eqref{probability function} and \eqref{self-similarity}, and $T$ is invariant under rotations of the pair $(\xi,\xi')$, then
  \begin{equation}
\label{probability function-2}
\int_{\R^3} T(\xi,\xi')\,d\xi'\leq \bar{K}<\infty.
\end{equation}
 \end{Lemma}
\begin{proof} The proof is motivated by the work of \cite{LV}.
From \eqref{same speed}, we  deduce
\begin{equation}
\label{big speed case}
T(\xi,\xi')=0\quad\text{ if } |\xi|>|\xi'|.
\end{equation}
By \eqref{probability function} and \eqref{big speed case}, we find
\begin{equation}
\begin{split}
\label{probability function-3}
1=\int_{\R^3}T(\xi,\xi')\,d\xi&=\int_{B(0,|\xi'|)}T(\xi,\xi')\,d\xi+\int_{|\xi|>|\xi'|}T(\xi,\xi')\,d\xi
\\&=\int_{B(0,|\xi'|)}T(\xi,\xi')\,d\xi,
\end{split}
\end{equation}
where $B(0,|\xi'|)$ is the set of points in the interior of a sphere of radius $ |\xi'|$, centered at $ 0$.
Note that, 
  \eqref{self-similarity} and \eqref{probability function-3}, one obtains
\begin{equation*}
\begin{split}
1&=H(|\xi'|)\int_{B(0,|\xi'|)}T(\frac{\xi}{|\xi'|},\frac{\xi'}{|\xi'|})\,d\xi
\\&=H(|\xi'|)\int_{B(0,1)}T(z,\frac{\xi'}{|\xi'|})|\xi'|^3\,dz
\\&=|\xi'|^3 H(|\xi'|),
\end{split}
\end{equation*}
where $z=\frac{\xi}{|\xi'|}.$
This gives us
\begin{equation}
\label{H value}
H(|\xi'|)=\frac{1}{|\xi'|^3},
\end{equation}
and hence \begin{equation*}
\begin{split}
\int_{\R^3}T(\xi,\xi')\,d\xi'&=\int_{B(0,|\xi'|)}H(|\xi'|)T(\frac{\xi}{|\xi'|},\frac{\xi'}{|\xi'|})\,d\xi'
\\&=\frac{1}{|\xi|^3}\int_{B(0,|\xi|)}T(\frac{\xi}{|\xi|},\frac{\xi'}{|\xi|})\,d\xi'
\\&=\frac{1}{|\xi|^3}\int_{B(0,1)}T(\frac{\xi}{|\xi|},\eta)|\xi|^3\,d\eta
\\&=\int_{B(0,1)}T(\frac{\xi}{|\xi|},\eta)\,d\eta
\\&\leq \bar{K},
\end{split}
\end{equation*}
where we used \eqref{same speed} and \eqref{H value}, $\bar{K}>0$ is a fixed number.
\end{proof}
With these two lemmas, we are ready to construct smooth solutions of a suitable approximation and pass to the limits to recover the weak solutions.
\vskip0.3cm

\section{Regularized equation}
The subjective of this section is to construct a sequence of smooth solutions to a regularized equation and to derive some uniformly estimates on them. In particular,
we construct a sequence of solutions verifying the following regularized equation
\begin{equation}
\label{regularized equation}
\begin{cases}
&(f_k)_t+\xi\cdot\nabla f_k=-\mu(\mathfrak{n}_{k-1})f_k-\lambda f_k+\lambda\int_{\R^3}T(\xi,\xi')f_{k-1}(t,x,\xi')\,d\xi',
\\& f_k(0,x,\xi)=f^0_{\varepsilon}(x,\xi),
\\& \mathfrak{n}_{k-1}=\int_{\R^3}f_{k-1}\,d\xi,
\\&f_0(t,x,\xi)=0,
\end{cases}
\end{equation}
where $k\geq 0$ are integers, $h_{\varepsilon}(x)=h*\eta_{\varepsilon}(x).$ In fact, a similar approximation in \cite{LV} motivated us to proposal the above ones.  For any given $f_{k-1}$, we can solve regularized equation \eqref{regularized equation} by the characteristic method.
In particular, the smooth solution of the following ODE:
\begin{equation}
\begin{cases}
\label{ode}
&\frac{dx}{dt}=\xi;\\
&x(0)=x
\end{cases}
\end{equation}
is given by $x(t)=x+t\xi.$  Thus, along the characteristic line $$x(t)=x+t\xi,$$ the initial value problem \eqref{regularized equation} corresponds to the following ODE:
\begin{equation}
\label{ODE system}
\begin{cases}
&\frac{d}{dt}f_k(t,x(t),\xi)=-\mu(\mathfrak{n}_{k-1})f_k(t,x(t),\xi)-\lambda f_k(t,x(t),\xi)
\\&\quad\quad\quad \quad\quad\quad \quad\quad\quad  +\lambda\int_{\R^3}T(\xi,\xi')f_{k-1}(t,x(t),\xi')\,d\xi',
\\& \mathfrak{n}_{k-1}=\int_{\R^3}f_{k-1}\,d\,\xi,
\\&f_k(0,x,\xi)=f_{\varepsilon}^0(x,\xi).
\end{cases}
\end{equation}
Our task is to solve $f_k$ for any given $f_{k-1}.$
By means of the classical theory of ODE, there exists a smooth solution to \eqref{regularized equation}
\begin{equation}
\label{regularized solution}
\begin{split}
f_k(t,x,\xi)&=e^{-\int_0^t\left(\mu(\mathfrak{n}_{k-1})(s)+\lambda\right)\,ds}f_{\varepsilon}^0(x+t\xi,\xi)
\\&+\int_0^t\int_{\R^3}e^{-\int_{\tau}^t\left(\mu(\mathfrak{n}_{k-1})(s)+\lambda\right)\,ds}T(\xi,\xi')f_{k-1}(\tau,x(\tau),\xi')\,d\xi'\,d\tau.
\end{split}
\end{equation}
This yields
\begin{equation*}
f_k\geq 0\quad\text{ for all } k\geq 0.
\end{equation*}

\bigskip

The following lemma gives us a general version of Gronwall inequality  for a sequence of the nonnegative continuous functions. It allows us to derive further uniformly estimates of $f_k$.  This Gronwall inequality was proved by induction in the paper of Boudin-Desvillettes-Grandmont-Moussa \cite{BDGM}. We rely on it to deduce our several key estimates in this current paper.
\begin{Lemma}
\label{Gronwall inequality}
Let $T>0,$ a sequence $\{a_n\}_{n=0}^{\infty}$ of nonnegative continuous function on $[0,T]$, for any $n\geq 0$, if
$$a_{n}(t)\leq A+B\int_0^ta_{n-1}(s)\,ds,\quad\text{ for any } 0\leq t\leq T,$$
then, there exists a constant $K\geq 0$ such that
$$a_n(t)\leq K e^{Kt}, \quad\text{ for any } 0\leq t\leq T,$$
where $$K=\max\{A,B,\sup_{0\leq t\leq T}a_0\}.$$
\end{Lemma}

\vskip0.3cm

In \eqref{regularized solution}, the value of $f_k$ depends on $f_{k-1}$, thus we apply Lemma \ref{Gronwall inequality} to get the following lemma on the estimates of $f_k$.
\begin{Lemma}
\label{Lemma bounds on fk}
If $f_k(t,x,\xi)$ is given by \eqref{regularized solution}, for any $k\geq 0$, then $f_k(t,x,\xi)$ is bounded in $$L^{\infty}(0,T;L^{\infty}(\R^3\times\R^3))\cap L^{\infty}(0,T;L^1(\R^3\times\R^3)),$$
and hence
\begin{equation}
\label{Lp bound of f}
f_k(t,x,\xi)\;\;\text{  is bounded in } L^{\infty}(0,T;L^{p}(\R^3\times\R^3))\;\;\text{ for any } p\geq 1.
\end{equation}
\end{Lemma}

\begin{proof}
Using \eqref{regularized solution} and \eqref{probability function-2}, we deduce
\begin{equation}
\begin{split}
\label{AA}
\|f_k(t,x,\xi)\|_{L^{\infty}}&\leq \|f_{\varepsilon}^0\|_{L^{\infty}}+\lambda \bar{K} \int_0^t\|f_{k-1}(\tau,x,\xi)\|_{L^{\infty}}\,d\tau.
\end{split}
\end{equation}
Applying Lemma \ref{Gronwall inequality} to \eqref{AA}, one obtains
$$\|f_k\|_{L^{\infty}}\leq \lambda K \bar{K} e^{Kt}\quad\text{ for any } 0\leq t\leq T,$$
where $$K=\max\{\|f_{\varepsilon}^0\|_{L^{\infty}},1,\|f_0\|_{L^{\infty}}\}.$$ Thus, $f_k(t,x,\xi)$ is bounded in $L^{\infty}(0,T;L^{\infty}(\R^3\times\R^3))$ for any $k\geq 0.$

By \eqref{ODE system},  for any $0\leq t\leq T$, we find
\begin{equation}
\begin{split}
\label{AAAA}
&\frac{d}{dt}\int_{\R^3\times\R^3}f_k(t,x(t),\xi)\,d\xi\,d x +\int_{\R^3\times\R^3}\mu(\mathfrak{n}_{k-1})f_k(t,x(t),\xi)\,dx\,d\xi
\\&+\lambda\int_{\R^3\times\R^3} f_k(t,x(t),\xi)\,dx\,d\xi
\\&= \lambda
\int_{\R^3\times\R^3}\int_{\R^3}T(\xi,\xi')f_{k-1}(t,x(t),\xi')\,d\xi'\,d\xi\,d x.
\end{split}
\end{equation}
The term on the right side of \eqref{AAAA} is given by
$$\lambda
\int_{\R^3\times\R^3}f_{k-1}(t,x(t),\xi)\,d\xi\,d x.$$
Thus, \eqref{AAAA} gives us
\begin{equation*}
\begin{split}
&\int_{\R^3\times\R^3}f_k\,d\xi\,d x\leq \int_{\R^3\times\R^3}f_{\varepsilon}^0\,d\xi\,d x
\\&\quad\quad\quad\quad\quad\quad\quad+\lambda \int_0^t\int_{\R^3\times\R^3}f_{k-1}(\tau,x(\tau),\xi)\,d\xi\,dx\,d\tau.
\end{split}
\end{equation*}
Thanks to Lemma \ref{Gronwall inequality}, there exists a constant $K>0$, for any $0\leq t\leq T$, such that
$$
\int_{\R^3\times\R^3}|f_k|\,d\xi\,d x\leq \lambda K e^{Kt}.$$
Thus, $f_k(t,x,\xi)$ is bounded in $L^{\infty}(0,T;L^{1}(\R^3\times\R^3)).$
\end{proof}

\vskip0.3cm
The estimates in the following Lemma \ref{Lemma of energy law for regularized equation} are crucial to control the bounds of the kinetic density and the kinetic current. The proof is also based on
 Lemma \ref{Gronwall inequality} again.
\begin{Lemma}
\label{Lemma of energy law for regularized equation} If $f_k(t,x,\xi)$ is given by \eqref{regularized solution}, and $$\int_{\R^3\times\R^3}(1+|\xi|^m)f^0_{\varepsilon}\,d\xi\,d x<+\infty$$ for some $m\geq 1,$ then, there exists a constant $K>0$, such that
\begin{equation}
\label{m order estimate}
\int_{\R^3\times\R^3}(1+|\xi|^m)f_k\,d\xi\,d x\leq \lambda Ke^{Kt},
\end{equation}
for any $0\leq t\leq T$.
\end{Lemma}

\begin{proof}
For some $m\geq 1$, using $1+|\xi|^m$ to multiply on
the both sides of \eqref{ODE system}, one obtains the following energy law
\begin{equation}
\label{energy law for regularized equation}
\begin{split}
\frac{d}{dt}\int_{\R^3\times\R^3}&(1+|\xi|^m)f_k\,d\xi\,d x+\int_{\R^3\times\R^3}\mu(\mathfrak{n}_{k-1})(1+|\xi|^m)f_k\,d\xi\,d x
\\&+\lambda\int_{\R^3\times\R^3}(1+|\xi|^m)f_k\,d\xi\,d x
\\&=\lambda\int_{\R^3\times\R^3\times\R^3}T(\xi,\xi')f_{k-1}(t,x,\xi')(1+|\xi|^m)\,d\xi'\,d\xi\,d x.
\end{split}
\end{equation}
Thanks to  \eqref{same speed}, \eqref{probability function}, and the Fubini's theorem, the right side term on \eqref{energy law for regularized equation} is given by
\begin{equation*}
\lambda \int_{\R^3\times\R^3}f_{k-1}(t,x,\xi)(1+|\xi|^m)\,d\xi\,dx.
\end{equation*}
This, with \eqref{energy law for regularized equation}, yields
\begin{equation}
\label{energy law for regularized equation-12}
\begin{split}
&\frac{d}{dt}\int_{\R^3\times\R^3}(1+|\xi|^m)f_k\,d\xi\,dx+\int_{\R^3\times\R^3}\mu(\mathfrak{n}_{k-1})(1+|\xi|^m)f_k\,d\xi\,dx
\\&+\lambda\int_{\R^3\times\R^3}(1+|\xi|^m)f_k\,d\xi\,dx
=
\lambda \int_{\R^3\times\R^3}f_{k-1}(t,x,\xi)(1+|\xi|^m)\,d\xi\,dx.
\end{split}
\end{equation}
Integrating on both sides of \eqref{energy law for regularized equation-12} with respect to $t$, one obtains
\begin{equation}
\label{energy inequality for regularized equation-1}
\begin{split}
&\int_{\R^3\times\R^3}(1+|\xi|^m)f_k\,d\xi\,dx
\leq \int_{\R^3\times\R^3}(1+|\xi|^m)f^0_{\varepsilon}\,d\xi\,dx
\\&+
\lambda \int_0^t\int_{\R^3\times\R^3}f_{k-1}(\tau,x(\tau,\xi),\xi)(1+|\xi|^m)\,d\xi\,dx\,d\tau.
\end{split}
\end{equation}
Thus, applying Lemma \ref{Gronwall inequality} to \eqref{energy inequality for regularized equation-1}, there exists a constant $K>0$, such that
\begin{equation}
\label{energy inequality for regularized equation}
\int_{\R^3\times\R^3}(1+|\xi|^m)f_k\,d\xi\,dx\leq \lambda Ke^{Kt}
\end{equation}
for any $0\leq t\leq T$ and any $k\geq 0,$ where $$K=\max\{\int_{\R^3\times\R^3}(1+|\xi|^m)f^0_{\varepsilon}\,d\xi\,dx,\;\; \lambda,\;\;
\sup_{0\leq t\leq T}\int_{\R^3\times\R^3}(1+|\xi|^m)f_0\,d\xi\,dx\}.$$
\end{proof}

\vskip0.3cm

\begin{Remark}
Lemma \ref{Lemma of energy law for regularized equation} allows us to deduce that the estimates on the kinetic density and the kinetic current. Those estimates, with the help of $L^p$ regularity of average velocities,  yield the strong convergence of them in $L^p$ space.
\end{Remark}
To our convenience, we introduce the kinetic density (zero moment)
$$\mathfrak{n}_k(t,x)=\int_{\R^N}f_k\,d\xi,$$
and the kinetic  current (first moment)
$$ j_k(t,x)=\int_{\R^N}\xi f_k\,d\xi$$
in the space $\R^N$ with respect to $\xi$.
\vskip0.3cm

 We estimate these quantities  in the following lemma \ref{Lemma of n-j}  that may be similar to the variation of the classical regularity of moments, see \cite{LP}. The estimates of  the kinetic density and the kinetic current help us to get the weak stability of kinetic equation.
\begin{Lemma}
\label{Lemma of n-j}For any $p\geq 1$, $0\leq t\leq T$, if $f_k$ is bounded in $L^{\infty}(\R^3\times\R^3\times[0,T]),$ we have\begin{equation*}
\|\mathfrak{n}_k(t,x)\|_{L^{\infty}(0,T;L^{\frac{N+p}{N}}(\R^N))}\leq C_{N,T}(\|f_k\|_{L^{\infty}}+1)\left(\int_{\R^N\times\R^N}|\xi|^pf_k\,d\xi\,dx\right)^{\frac{N}{N+p}},
\end{equation*}
$$\|j_k\|_{L^{\infty}(0,T;L^{\frac{N+p}{N+1}}(\R^N))} \leq C_{N,T}(\|f_k\|_{L^{\infty}}+1)\left(\int_{\R^N\times\R^N}|\xi|^pf_k\,d\xi\,dx\right)^{\frac{N+1}{N+p}}.$$

\end{Lemma}
\begin{proof} The proof is following the same line in the work of Hamdache \cite{H}.  For any $R>0$,
we estimate $\mathfrak{n}_k$ as follows
\begin{equation}
\label{split n}
\begin{split}
\mathfrak{n}_k(t,x)&=\int_{\R^N} f_k\,d\xi=\int_{|\xi|\leq R} f_k\,d\xi+\int_{|\xi|\geq R}f_k\,d\xi
\\&\leq C_NR^N\|f_k\|_{L^{\infty}}+\frac{1}{ R^p}\int_{|\xi|\geq R}|\xi|^pf_k\,d\xi.
\end{split}
\end{equation}
Taking
 $$R=\left(\int_{\R^N}|\xi|^pf_k\,d\xi\right)^{\frac{1}{N+p}},$$
  one obtains
\begin{equation*}
\mathfrak{n}_k(t,x)\leq C_{N}(\|f_k\|_{L^{\infty}}+1)\left(\int_{\R^N}|\xi|^pf_k\,d\xi\right)^{\frac{N}{N+p}}.
\end{equation*}
This yields
\begin{equation*}
\|\mathfrak{n}_k(t,x)\|_{L^{\infty}(0,T;L^{\frac{N+p}{N}}(\R^N))}\leq C_{N,T}(\|f_k\|_{L^{\infty}}+1)\left(\int_{\R^N\times\R^N}|\xi|^pf_k\,d\xi\,dx\right)^{\frac{N}{N+p}},
\end{equation*}
where $f_k$ is bounded in $L^{\infty}$ due to \eqref{Lp bound of f}.\\
Following the same arguments, we can show
$$\|j_k\|_{L^{\infty}(0,T;L^{\frac{N+p}{N+1}}(\R^N))} \leq C_{N,T}(\|f_k\|_{L^{\infty}}+1)\left(\int_{\R^N\times\R^N}|\xi|^pf_k\,d\xi\,dx\right)^{\frac{N+1}{N+p}}.$$
\end{proof}

\vskip0.3cm
On one hand, in three dimensional space,
 Lemma \ref{Lemma of n-j} implies
\begin{equation}
\label{nk estimate}
\|\mathfrak{n}_k(t,x)\|_{L^{\infty}(0,T;L^{\frac{3+m}{3}}(\R^3))}\leq C_{N}(\|f_k\|_{L^{\infty}}+1)\left(\int_{\R^3\times\R^3}|\xi|^mf_k\,d\xi\,dx\right)^{\frac{3}{3+m}},
\end{equation}
\begin{equation}
\label{jk estimate}
\|j_k(t,x)\|_{L^{\infty}(0,T;L^{\frac{3+m}{4}}(\R^3))}\leq C_{N}(\|f_k\|_{L^{\infty}}+1)\left(\int_{\R^3\times\R^3}|\xi|^mf_k\,d\xi\,dx\right)^{\frac{4}{3+m}}.
\end{equation}

On the other hand, by Lemma \ref{Lemma of energy law for regularized equation},  there exists a constant $K>0$, such that
\begin{equation}
\label{kinetic energy}
\int_{\R^3\times\R^3}(1+|\xi|^m)f_k\,d\xi\,dx< K e^{Kt},
\end{equation}
for any $0\leq t \leq T,$ where $K>0$ only depends on the initial data.

Thus, the right sides of \eqref{nk estimate}-\eqref{jk estimate} can be controlled by the initial data, thanks to \eqref{kinetic energy}. This yields  the kinetic density
\begin{equation}
\label{n Lp estimate}
\mathfrak{n}_k(t,x)\quad\text{is bounded in } L^{\infty}(0,T;L^r(\R^3))\;\;\text{ for any } 1\leq r\leq \frac{3+m}{3},
 \end{equation}
 and the kinetic current
 \begin{equation}
 \label{j Lp estimate}
 j_k(t,x)\quad\text{ is bounded in } L^{\infty}(0,T;L^s(\R^3))\;\;\text{ for any } 1\leq s\leq \frac{3+m}{4}.
 \end{equation}


\vskip0.3cm

Thus, we have proved the following proposition on the existence of  solution to initial value problem \eqref{regularized equation} in this section.
\begin{Proposition}
\label{prop regularized equation}
For any given $\varepsilon>0$, $k\geq0$ and $T>0$, under the assumption of Theorem \ref{main result},  there exists a smooth solution to initial value problem \eqref{regularized equation} which is given by \eqref{regularized solution}. Moreover, the solution satisfies the energy equality \eqref{energy law for regularized equation}, and the estimates of \eqref{Lp bound of f}, \eqref{m order estimate}, \eqref{n Lp estimate}, and \eqref{j Lp estimate}.
\end{Proposition}
\begin{Remark}
The solution constructed in above Proposition \ref{prop regularized equation} is a smooth solution, it obeys the energy equality \eqref{energy law for regularized equation-12}.
\end{Remark}

\bigskip\bigskip

\section{Recover the weak solutions}
The proof of Theorem \ref{main result} will be developed in this current section. It relies on the
introduction of a regularized equation \eqref{regularized equation} (for which the existence of a
solutions is given in Proposition \ref{prop regularized equation}, in fact, it is a smooth solution for any given $k>0$ and $\varepsilon>0$). In order to  recover the weak solution to \eqref{kinetic equation}-\eqref{initial data}, we shall pass to the limits as $k$ goes to large and $\varepsilon$ tends to zero, show that the limit function is a weak solution to initial value problem \eqref{kinetic equation}-\eqref{initial data}.  Thus, we need some convergence on the function $f_{k,\varepsilon}(t,x,\xi)$,  the particles density (zero moment)
$\mathfrak{n}_{k,\varepsilon}(t,x)$
and the kinetic  current (first moment) $j_{k,\varepsilon}(t,x)$ in $L^p$ space for some $p>1.$  In the following subsections, we will handle the limits with respect to $k$ in Subsection 4.1 and pass to the limits with respect to $\varepsilon$ in Subsection 4.2.

\subsection{Passing to the limits as $k\to\infty$} In this subsection, we use $\{f_k\}_{k=0}^{\infty}$ to denote the sequence of solutions to \eqref{regularized equation} that constructed in Proposition \ref{prop regularized equation} for any fixed $\varepsilon>0$. Here we aim at passing to the limits as $k$ goes to infinity for any given $\varepsilon>0$.

\vskip0.3cm

The estimates of Lemma \ref{Lemma bounds on fk} and Lemma \ref{Lemma of energy law for regularized equation} are crucial ones in this subsection.
In fact, these solutions satisfy, for all $T>0$,
\begin{equation}
\begin{cases}
\label{a priori estimate on k}
&\|f_k\|_{L^{\infty}(0,T;L^p(\R^3\times\R^3))}\leq C,\quad\text{ for any } 1\leq p\leq +\infty,
\\&
\|\mathfrak{n}_k(t,x)\|_{L^{\infty}(0,T;L^{r}(\R^3))}\leq C,\;\;\text{ for any } 1\leq r\leq \frac{3+m}{3},
\\&
\|j_k(t,x)\|_{L^{\infty}(0,T;L^s(\R^3))}\leq C,\;\;\text{ for any } 1\leq s\leq \frac{3+m}{4},
\end{cases}
\end{equation}
where all $C>0$ only depend on the initial data.

\vskip0.3cm

First of all, by \eqref{a priori estimate on k}, it follows that there exists a function $f(t,x,\xi)$ such that
\begin{equation}
\label{weak convergence of f}
f_k\rightharpoonup f\quad\text{ weakly in }L^{\infty}(0,T;L^p(\R^3\times\R^3)),\;\;\text{ for any } p> 1.
\end{equation}
In particular, this limit function $f(t,x,\xi)$ is bounded in $L^{\infty}(0,T;L^{\infty}(\R^3\times\R^3).$\\

\vskip0.3cm

Due to nonlinear term $\mu(\mathfrak{n})f$ in \eqref{kinetic equation}, or the nonlinear term $\mu(\mathfrak{n}_{k-1})f_k$ in \eqref{regularized equation},
we shall study the strong convergence of $\mathfrak{n}_k$ and $j_k$ in some $L^p$ space in the following Lemma \ref{Lemma on strong convergence}.
\begin{Lemma} Let $f_k$ be a solution to \eqref{regularized equation} constructed in Proposition \ref{prop regularized equation}, then
\label{Lemma on strong convergence}

\begin{equation}
\label{first strong convergence}
 \mathfrak{n}_k\to \mathfrak{n} \;\text{ strongly in } L^{\infty}(0,T;L_{loc}^{r}(\R^3)),
\end{equation}
\begin{equation}
\label{second strong convergence}
j_k\to j  \;\text{ strongly in } L^{\infty}(0,T;L_{loc}^{s}(\R^3)),
\end{equation}
\begin{equation}
\label{strong convergence kinetic energy-1}
\int_{\R^3}(1+|\xi|^m)f_k\,d\xi\to \int_{\R^3}(1+|\xi|^m)f\,d\xi  \;\text{ strongly in } L^{\infty}(0,T;L_{loc}^{1}(\R^3)).
\end{equation}
\end{Lemma}
\begin{proof} Thanks to \eqref{m order estimate} with $m=2$,
 $f_k$ is bounded in $$L^{\infty}(0,T;L^1(\R^3\times\R^3),1+|\xi|^2).$$
By \eqref{m order estimate} and \eqref{energy law for regularized equation-12} for any $m\geq 1$, one obtains
\begin{equation}
\label{energy damping}
\int_0^T\int_{\R^3\times\R^3}(1+|\xi|^m)\mu(\mathfrak{n}_{k-1})f_k\,d\xi\,d\x\,dt\leq C.
\end{equation}
This yields
\begin{equation}
\label{Q1 in L^1}
\mu(\mathfrak{n}_{k-1})f_k \;\text{ is bounded in } L^1(0,T;L^1(\R^3\times\R^3)).
\end{equation}
Meanwhile, the Fubini's Theorem gives us
\begin{equation}
\label{Q2 in L infty}
-\lambda f_k+\lambda\int T(\xi,\xi')f_{k-1}(t,x,\xi')\,d\xi' \text{ is bounded in } \; L^{\infty}(0,T;L^p(\R^3\times\R^3)),
\end{equation}
thanks to \eqref{a priori estimate on k}.

 By \eqref{Q1 in L^1} and \eqref{Q2 in L infty}, we deduce from Diperna-Lions-Meyer' $L^p$ regularity of average velocities, (see \cite{DLM}), for each $\phi(\xi)\in \mathfrak{D}(\R^3)$, $$\left\{\int_{\R^3} f^k\phi(\xi)\,d\xi\right\}_{k}$$ is relatively compact in $L^1(0,T;L^1(B_R)),$
for any $R<+\infty.$ In particular, we find
\begin{equation*}
\mathfrak{n}_k=\int_{\R^3}f^k\,d\xi\to \mathfrak{n}=\int_{\R^3}f\,d\xi,\;\;\text{a.e.}
\end{equation*}

\begin{equation*}
j_k=\int_{\R^3}\xi f^k\,d\xi\to j=\int_{\R^3}\xi f\,d\xi,\;\;\text{a.e.}
\end{equation*}
and
\begin{equation}
\label{convergence m order}
\int_{\R^3}(1+|\xi|^m)f_k\,d\xi\to \int_{\R^3}(1+|\xi|^m)f\,d\xi  \;\text{a.e.}
\end{equation}
Thanks to the estimates of \eqref{a priori estimate on k} for $\mathfrak{n}_k$ and $j_k$, we strength the above convergence of $\mathfrak{n}_k$ and $j_k$ as follows

\begin{equation*}
\label{n strong convergence}
 \mathfrak{n}_k\to \mathfrak{n} \;\text{ strongly in } L^{\infty}(0,T;L^{r}_{loc}(\R^3)),
\end{equation*}
and
\begin{equation*}
\label{j strong convergence}
j_k\to j  \;\text{ strongly in } L^{\infty}(0,T;L^{s}_{loc}(\R^3)).
\end{equation*}
Using \eqref{kinetic energy}, we  strength the above convergence of $
\int_{\R^3}(1+|\xi|^m)f_k\,d\xi$ as follows
\begin{equation*}
\label{strong convergence kinetic energy}
\int_{\R^3}(1+|\xi|^m)f_k\,d\xi\to \int_{\R^3}(1+|\xi|^m)f\,d\xi  \;\text{ strongly in } L^{\infty}(0,T;L^{1}_{loc}(\R^3)).
\end{equation*}
\end{proof}
\vskip0.3cm

With \eqref{n strong convergence} at hand, we are ready  to show  the convergence of $\mu(\mathfrak{n}_{k-1})\mathfrak{n}_k$ in the sense of distributions on $((0,T)\times\R^3)$. We address it
 in the following lemma.
\begin{Lemma}
\label{Lemma of Q1}
If $f_k$ is given by \eqref{regularized solution}, then \begin{equation}
\label{Q1 convergence}
\int_{\R^3}\mu(\mathfrak{n}_{k-1})f_k\,d \xi\to \int_{\R^3}\mu(\mathfrak{n})f\,d \xi
\end{equation}
in $L^{\infty}(0,T;L_{loc}^{\frac{r}{2}}(\R^3))$
as $k\to\infty.$ Moreover, if $m\geq 3,$ then $\frac{r}{2}\geq 1.$
\end{Lemma}

\begin{proof}

Firstly, \begin{equation}
\label{damping different form}
\begin{split}&
\int_{\R^3\times\R^3}\mu(\mathfrak{n}_{k-1})f_k\,d\xi\,d x=\int_{\R^3}\mu(\mathfrak{n}_{k-1})\mathfrak{n}_k\,dx,
\\&\int_{\R^3\times\R^3}\mu(\mathfrak{n})f\,d\xi\,d x=\int_{\R^3}\mu(\mathfrak{n})\mathfrak{n}\,dx.
\end{split}
\end{equation}

Since $\mu(\cdot)$ is a Liptschitz function and  \eqref{first strong convergence}, then
\begin{equation}
\label{strong convergence of mu}
\mu(\mathfrak{n}_{k-1})\to \mu(\mathfrak{n})\;\text{strongly in } L^{\infty}(0,T;L^r_{loc}(\R^3)).
\end{equation}
Thanks to  \eqref{first strong convergence}, \eqref{damping different form} and \eqref{strong convergence of mu}, we find, for any $m\geq 3,$
\begin{equation*}
\label{operator convergence}
\mu(\mathfrak{n}_{k-1})\mathfrak{n}_k\to \mu(\mathfrak{n})\mathfrak{n}
\end{equation*}
in $L^{\infty}(0,T;L_{loc}^{\frac{r}{2}}(\R^3))$
as $k\to\infty,$ and $\frac{r}{2}\geq 1.$
\end{proof}

\vskip0.3cm
To pass to the limits with respect to $k$, we still need the following lemma on the convergence of $$\mathfrak{Q}_2(f_k)=-\lambda f_k+\lambda\int_{\R^3} T(\xi,\xi')f_{k-1}(t,x,\xi')\,d\xi'$$ in some $L^p$ space.

\begin{Lemma}
Let $f_k$ be given by \eqref{regularized solution}, then
$\mathfrak{Q}_2(f_k)$ is bounded in $$L^{\infty}(0,T;L^p(\R^3\times\R^3)$$ for any $p\geq 1,$
and
\begin{equation}
\label{convergence of Q2}
\begin{split}&\mathfrak{Q}_2(f_k)\to \mathfrak{Q}_2(f)\quad\text
{ weakly in } L^{\infty}(0,T;L^p(\R^3\times\R^3))\;\;\text{ for any } p>1,
\end{split}
\end{equation}
 as $k\to\infty.$
\end{Lemma}
\begin{proof}
The proof follows the same line as in  \cite{LV} and we sketch it just for sake of completeness.
We estimate
\begin{equation}
\label{L infty bound of Q2}
\begin{split}
\|\mathfrak{Q}_2(f_k)\|_{L^{\infty}}&\leq \lambda\|f_k\|_{L^{\infty}}+\lambda\int_{\R^3} T(\xi,\xi')\,d\xi'\|f_{k-1}\|_{L^{\infty}}
\\&\leq \lambda\|f_k\|_{L^{\infty}}+\lambda K\|f_{k-1}\|_{L^{\infty}},
\end{split}
\end{equation}
and \begin{equation*}
\begin{split}
&\int_{\R^3\times\R^3}|\mathfrak{Q}_2(f_k)|\,d\xi\,dx=\int_{\R^3\times\R^3}\left|-\lambda f_k+\lambda\int T(\xi,\xi')f_{k-1}(t,x,\xi')\,d\xi'\right|\,d\xi\,dx
\\&\leq \lambda\int_{\R^3\times\R^3}f_k\,d\xi\,dx+\lambda\int_{R^3\times\R^3\times\R^3}T(\xi,\xi')f_{k-1}(t,x,\xi')\,d\xi'\,d\xi\,dx
\\& \leq \lambda\int_{\R^3\times\R^3}f_k\,d\xi\,dx+\lambda\int_{\R^3\times\R^3}f_{k-1}\,d\xi\,dx,
\end{split}
\end{equation*}
where we used $f_k(t,x,\xi)\geq 0$ for all $k$ and \eqref{probability function}. Thus,
\begin{equation}
\label{L1 bound of Q2}
\|\mathfrak{Q}_2(f_k)\|_{L^{\infty}(0,T;L^1(\R^3\times\R^3))}\leq C\|f_{k}\|_{L^{\infty}(0,T;L^1(\R^3\times\R^3))}.
\end{equation}
Using \eqref{L infty bound of Q2} and \eqref{L1 bound of Q2},
$\mathfrak{Q}_2(f_k)$ is bounded in $$L^{\infty}(0,T;L^p(\R^3\times\R^3)$$ for any $p\geq 1.$

For any $\varphi(x)\in L^{p}(0,T;L^q(\R^3))$, we consider
\begin{equation}
\label{the second part of Q2}
\begin{split}
&\int_{\R^3\times\R^3}\left(\int T(\xi,\xi')f_{k-1}(t,x,\xi')\,d\xi'-\int T(\xi,\xi')f(t,x,\xi')\,d\xi'\right)\varphi(x)\,d\xi\,dx\\
&=\int_{\R^3\times\R^3\times\R^3}T(\xi,\xi')\left(f_{k-1}(t,x,\xi')-f(t,x,\xi')\right)\varphi(x)\,d\xi'\,d\xi\,dx
\\&=\int_{\R^3\times\R^3}\left(f_{k-1}(t,x,\xi)-f(t,x,\xi)\right)\varphi(x)\,d\xi\,dx\to 0
\end{split}
\end{equation}
as $k\to\infty,$
thanks to \eqref{weak convergence of f}.

With the help of \eqref{weak convergence of f}, \eqref{L1 bound of Q2} and \eqref{the second part of Q2}, $\mathfrak{Q}_2(f_k)$ converges to $\mathfrak{Q}(f)$ weakly in
$$L^{\infty}(0,T;L^p(\R^3\times\R^3))$$ for any $p> 1.$
\end{proof}

\vskip0.3cm

Turning to the next issue, the smooth solution $f_k$ of \eqref{regularized equation} satisfies the following weak formulation
\begin{equation}
\begin{split}
\label{weak formulation on k}
&\int_{\R^3\times\R^3}f_{\varepsilon}^0\varphi(0,x,\xi)\,d\xi\,dx+\int_0^T\int_{\R^3\times\R^3}f_k\varphi_t\,d\xi\,dx\,dt+\xi f_k\cdot\nabla_x\varphi\,d\xi\,dx\,dt+
\\&\int_0^T\int_{\R^3\times\R^3}\mu(\mathfrak{n}_{k-1})f_k\varphi\,d\xi\,dx\,dt=-\lambda\int_0^T\int_{\R^3\times\R^3}f_k\varphi\,d\xi\,dx\,dt
\\&+\lambda\int_0^T\int_{\R^3\times\R^3\times\R^3}T(\xi,\xi')f_{k-1}(t,x,\xi')\varphi\,d\xi\,dx\,dt,
\end{split}
\end{equation}
where $$\mathfrak{n}_{k-1}=\int_{\R^3}f_{k-1}\,d\xi,$$
and $\varphi\in C^{\infty}(\R^3\times\R^3\times[0,T])$ is any test function.

Letting $k$ tends to infinity in \eqref{weak formulation on k}, using the above convergence, in particularly, by \eqref
{weak convergence of f}, \eqref
{first strong convergence}, \eqref{second strong convergence}, \eqref{Q1 convergence}, \eqref{convergence of Q2}, one obtains
\begin{equation}
\begin{split}
\label{weak formulation k}
&\int_{\R^3\times\R^3}f_{\varepsilon}^0\varphi(0,x,\xi)\,d\xi\,dx+\int_0^T\int_{\R^3\times\R^3}f\varphi_t+\xi f\cdot\nabla_x\varphi\,d\xi\,dx\,dt+
\\&\int_0^T\int_{\R^3\times\R^3}\mu(\mathfrak{n})f\varphi\,d\xi\,dx\,dt=-\lambda\int_0^T\int_{\R^3\times\R^3}f\varphi\,d\xi\,dx\,dt
\\&+\lambda\int_0^T\int_{\R^3\times\R^3\times\R^3}T(\xi,\xi')f(t,x,\xi')\varphi\,d\xi\,dx\,dt,
\end{split}
\end{equation}
and $$\mathfrak{n}=\int_{\R^3}f\,d\xi.$$

\vskip0.3cm

Concerning the energy inequality, it is reasonable to expect that
a priori estimate
we present can be further developed. Taking integration on both sides of \eqref{energy law for regularized equation-12} with respect to
$t$, one obtains
\begin{equation}
\label{energy law with m=2}
\begin{split}
&\int_{\R^3\times\R^3}(1+|\xi|^m)f_k\,d\xi\,dx+\int_0^t\int_{\R^3\times\R^3}\mu(\mathfrak{n}_{k-1})(1+|\xi|^m)f_k\,d\xi\,dx\,dt
\\&\quad\quad\quad\quad\quad\quad\quad\quad\quad\quad\quad=-\lambda\int_0^t\int_{\R^3\times\R^3}(1+|\xi|^m)f_k\,d\xi\,dx\,dt
\\&+
\lambda\int_0^t\int_{\R^3\times\R^3}f_{k-1}(1+|\xi|^m)\,d\xi\,dx\,dt+\int_{\R^3\times\R^3}(1+|\xi|^m)f^0_{\varepsilon}\,d\xi\,dx.
\end{split}
\end{equation}
Thanks to \eqref{strong convergence kinetic energy-1} and \eqref{convergence m order}, the Fatou's Lemma yields
\begin{equation*}
\int_{\R^3\times\R^3}(1+|\xi|^m)f\,d\xi\,dx\leq \lim_{k}\inf\int_{\R^3\times\R^3} (1+|\xi|^m)f_k\,d\xi\,dx.
\end{equation*}


By \eqref{m order estimate} and \eqref{convergence m order},
 the sum of the first two terms on the right side of \eqref{energy law with m=2} is zero.

Letting $k$ goes to infinite in \eqref{energy law with m=2},
 we have
 \begin{equation*}
\begin{split}
\int_{\R^3\times\R^3}(1+|\xi|^m)f\,d\xi\,dx&\leq \int_{\R^3\times\R^3}(1+|\xi|^m)f^0_{\varepsilon}\,d\xi\,dx,
\end{split}
\end{equation*}
for any $0\leq t\leq T.$

Thus, we have proved the following existence of weak solutions in this subsection by letting $k$ tends to infinity.
\begin{Proposition}
\label{prop k goes to large}
For any given $\varepsilon>0$, under assumption of Theorem \ref{main result}, there exists a weak solution for any $T>0$ to the following initial value problem
\begin{equation}
\label{regularized equation after k goes to large}
\begin{cases}
&f_t+\xi\cdot\nabla f=-\mu(\mathfrak{n})f-\lambda f+\lambda\int T(\xi,\xi')f(t,x,\xi')\,d\xi',
\\& f(0,x,\xi)=f^0_{\varepsilon}(x,\xi),
\\& \mathfrak{n}=\int_{\R^3}f\,d\xi,
\end{cases}
\end{equation}
where $h_{\varepsilon}(x)=h*\eta_{\varepsilon}(x).$
Moveover, the weak solution $f(t,x,\xi)$ has the following energy inequality
 \begin{equation}
\begin{split}
\label{energy law after k}
\int_{\R^3\times\R^3}(1+|\xi|^m)f\,d\xi\,dx \leq \int_{\R^3\times\R^3}(1+|\xi|^m)f^0_{\varepsilon}\,d\xi\,dx,
\end{split}
\end{equation}
for any $0\leq t\leq T,$
and \begin{equation*}
\label{bounded in Lp for f}
\|f\|_{L^{\infty}(0,T;L^p(\R^3\times\R^3))}\leq C,
\end{equation*}
for any $1\leq p\leq +\infty,$
where $C$ only depends on the initial data.
\end{Proposition}

\bigskip

\subsection{Passing to the limits as $\varepsilon\to 0$.}
In this subsection, we use $\{f_{\varepsilon}\}_{\varepsilon>0}$ to denote a sequence of solutions to initial value problem \eqref{regularized equation after k goes to large} that constructed in Proposition \ref{prop k goes to large} for any $\varepsilon>0$. Here we aim to pass to the limits for recovering the weak solutions to \eqref{kinetic equation}-\eqref{initial data} as $\varepsilon$ goes to zero.

On one hand, by Proposition \ref{prop k goes to large}, these solutions have
 the following estimate
 \begin{equation}
\label{bounded in Lp for f}
\|f_{\varepsilon}\|_{L^{\infty}(0,T;L^p(\R^3\times\R^3))}\leq C,
\end{equation}
for any $1\leq p\leq +\infty,$
where $C$ only depends on the initial data.
This yields,
\begin{equation}
\label{weak convergence last level}
f_{\varepsilon}\rightharpoonup f\quad\text{ weakly in } L^{\infty}(0,T;L^p(\R^3\times\R^3))
\end{equation}
for any $1<p<\infty.$

\vskip0.3cm

On the other hand, the solutions obey the following energy inequality,
 \begin{equation}
\label{energy inequality last level-1}
\begin{split}
\int_{\R^3\times\R^3}(1+|\xi|^m)f_{\varepsilon}\,d\xi\,d x
\leq \int_{\R^3\times\R^3}(1+|\xi|^m)f^0_{\varepsilon}\,d\xi\,d x,
\end{split}
\end{equation}
for all $\varepsilon>0.$

By the definition of $f^0_{\varepsilon}$, thus
$$f_{\varepsilon}^0\to f^0\;\;\text{a.e. as } \varepsilon\to0,\text{ and } f_{\varepsilon}^0\to f^0 \;\;\text{ in } L^p(\R^3\times\R^3) $$
for any $p> 1.$ Let us to denote $\overline{g}(t,x,\xi)=g*\eta_{\varepsilon}(x),$
where $\{\eta_{\varepsilon}\}_{\varepsilon>0}$ is a suitable family of regularizing kernels
with respect to the space variable $x$.
We state the following lemma:
\begin{Lemma}
\label{smooth lemma}
Let $\xi$ be the third variable, then, for any function $h=h(\xi)$, we have
\begin{equation*}
\int_{\R^3} \overline{g} h(\xi)\,d\xi=\overline{\int_{\R^3}gh\,d\xi}.
\end{equation*}
\end{Lemma}
\begin{proof}
Let us to calculate the left hand side,
\begin{equation*}
\begin{split}&
\int_{\R^3}\overline{g}h(\xi)\,d\xi=\int_{\R^3}\int_{\R^3}\eta_{\varepsilon}(y-x) g(t, y,\xi)h(\xi)\,dy\,d\xi,
\end{split}
\end{equation*}
and the right hand side is as follows
\begin{equation*}
\begin{split}&
\overline{\int_{\R^3}g(t,x,\xi)h(\xi)\,d\xi}=\int_{\R^3}\eta_{\varepsilon}(y-x)\left(\int_{\R^3} g( y,\xi)h(\xi)\,d\xi\right)\,dy.
\end{split}
\end{equation*}
By the Fubini's theorem, we proved this lemma.
\end{proof}

We use $\overline{f_0}=f^0_{\varepsilon}$, by Lemma \ref{smooth lemma},
we find that
$$\int_{\R^3\times\R^3}(1+|\xi|^m)f^0_{\varepsilon}\,d\xi\,d x=\int_{\R^3}\overline{\int_{\R^3}(1+|\xi|^m)f^0\,d\xi}\,d x.$$
Note that, $$\int_{\R^3}(1+|\xi|^m)f^0\,d\xi$$ is uniformly bounded in $L^{1}(\R^3),$ thus, for any $R\in(0,\infty)$,
\begin{equation}
\label{limit before R}
\int_{B_R}\overline{\int_{\R^3}(1+|\xi|^m)f^0\,d\xi}\,d x
\to \int_{B_R}\int_{\R^3}(1+|\xi|^m)f^0\,d\xi\,d x,
\end{equation}
where $B_R$ is a ball with radius $R>0$ and center at zero in $\R^3.$ \eqref{limit before R} is true for any $R>0,$ thus Letting $R$ go to $+\infty$ in \eqref{limit before R},  we have
\begin{equation*}
\int_{\R^3}\overline{\int_{\R^3}(1+|\xi|^m)f^0\,d\xi}\,d x
\to \int_{\R^3}\int_{\R^3}(1+|\xi|^m)f^0\,d\xi\,d x,
\end{equation*}

and hence,
as $\varepsilon$ tends to zero,
\begin{equation}
\label{initial data convergence}
 \int_{\R^3\times\R^3}(1+|\xi|^m)f^0_{\varepsilon}\,d\xi\,dx\to
\int_{\R^3\times\R^3}(1+|\xi|^m)f^0\,d\xi\,dx
\end{equation}
as $\varepsilon\to 0$.
Using \eqref{initial data convergence}, we deduce
\begin{equation}
\label{energy inequality last level}
\begin{split}
\int_{\R^3\times\R^3}(1+|\xi|^m)f_{\varepsilon}\,d\xi\,d x\leq \int_{\R^3\times\R^3}(1+|\xi|^m)f^0\,d\xi\,d x<+\infty,
\end{split}
\end{equation}
for all $\varepsilon>0.$

Thanks to \eqref{energy inequality last level}, we employ the same argument as in last subsection, to have, for any $\varepsilon\to 0$,

\begin{equation}
\label{weak convergence of n in Lp-2}
\mathfrak{n}_{\varepsilon}\to\mathfrak{n} \quad\text{ strongly in } L^{\infty}(0,T;L_{loc}^{r}(\R^3)),
\end{equation}
where $\mathfrak{n}=\int_{\R^3}f\,d\xi;$
 \begin{equation}
\label{weak convergence of j in Lp-2}
j_{\varepsilon}\to j \quad\text{ strongly in } L^{\infty}(0,T;L^{s}_{loc}(\R^3)),
\end{equation}
where $j=\int_{\R^3}f\xi\,d \xi;$
and
\begin{equation}
\label{operator convergence-2}
\int_{\R^3}\mu(\mathfrak{n}_{\varepsilon})f_{\varepsilon}\,d\xi\to \int_{\R^3}\mu(\mathfrak{n})f\,d\xi
\end{equation} in $L^{\infty}(0,T;L^1_{loc}(\R^3)).$
From \eqref{weak convergence last level}, as $\varepsilon\to 0$, one obtains
\begin{equation}
\label{Q2 convergence last level}
\begin{split}
&\mathfrak{Q}_2(f_{\varepsilon})=-\lambda f_{\varepsilon}+\lambda\int T(\xi,\xi')f_{\varepsilon}(t,x,\xi')\,d\xi'\to
\\&\quad\quad\quad\quad\mathfrak{Q}_2(f)=-\lambda f+\lambda\int T(\xi,\xi')f(t,x,\xi')\,d\xi'
\end{split}
\end{equation}  weakly in $ L^{\infty}(0,T;L^p(\R^3\times\R^3))$
for any $p\geq1$.
Thus, we can pass to the limits for recovering the weak solutions to \eqref{kinetic equation}-\eqref{initial data} as $\varepsilon$ tends to zero. In fact, by \eqref{weak convergence last level}, \eqref{weak convergence of n in Lp-2}-\eqref{Q2 convergence last level}, taking the limits as $\varepsilon$ tends to zero, in the following weak formulation,
\begin{equation}
\begin{split}
\label{weak formulation k}
&\int_{\R^3\times\R^3}f_{\varepsilon}^0\varphi(0,x,\xi)\,d\xi\,dx+\int_0^T\int_{\R^3\times\R^3}f_{\varepsilon}\varphi_t+\xi f_{\varepsilon}\cdot\nabla_x\varphi\,d\xi\,dx\,dt+
\\&\int_0^T\int_{\R^3\times\R^3}\mu(\mathfrak{n}_{\varepsilon})f_{\varepsilon}\varphi\,d\xi\,dx\,dt=-\lambda\int_0^T \int_{\R^3\times\R^3}f_{\varepsilon}\varphi\,d\xi\,dx\,dt
\\&\quad\quad\quad\quad\quad\quad\quad\quad\quad+\int_0^T\int_{\R^3\times\R^3}\left(\lambda\int T(\xi,\xi')f_{\varepsilon}(t,x,\xi')\,d\xi'\right)\,\varphi\,d\xi\,dx\,dt,
\end{split}
\end{equation}
where $$\mathfrak{n}_{\varepsilon}=\int_{\R^3}f_{\varepsilon}\,d\xi;$$
then
\begin{equation}
\begin{split}
\label{weak formulation}
&\int_{\R^3\times\R^3}f^0\varphi(0,x,\xi)\,d\xi\,dx+\int_0^T\int_{\R^3\times\R^3}f\varphi_t+\xi f\cdot\nabla_x\varphi\,d\xi\,dx\,dt+
\\&\int_0^T\int_{\R^3\times\R^3}\mu(\mathfrak{n})f\varphi\,d\xi\,dx\,dt=-\lambda\int_0^T\int_{\R^3\times\R^3}f\varphi\,d\xi\,dx\,dt
\\&\quad\quad\quad\quad\quad\quad\quad\quad\quad+\lambda\int_0^T\int_{\R^3\times\R^3}\left(\int_{\R^3}T(\xi,\xi')f(t,x,\xi')\,d\xi'\right)\varphi\,d\xi\,dx\,dt,
\end{split}
\end{equation}
where $$\mathfrak{n}=\int_{\R^3}f\,d\xi.$$

\vskip0.3cm

Moreover, same to last subsection, letting $\varepsilon\to0$ in \eqref{energy inequality last level}, we have
\begin{equation*}
\label{energy inequality after k going to infinity}
\begin{split}
&\int_{\R^3\times\R^3}(1+|\xi|^m)f\,d\xi\,dx\leq \int_{\R^3\times\R^3}(1+|\xi|^m)f^0\,d\xi\,dx.
\end{split}
\end{equation*}
Thus, we proved Theorem \ref{main result}.

\section{Energy conservation}

In this section, we will prove our second main result on the energy conservation.
 We can use the following quantities $$\Phi(t,x,\xi)=\eta_{\varepsilon}(y-x)\phi(\xi)$$
as a test function in \eqref{weak formulation in definition}, where $\{\eta_{\varepsilon}\}_{\varepsilon>0}$ is a suitable family of regularizing kernels
with respect to the space variable $x$.
Using $\overline{g}=g*\eta_{\varepsilon}$, we deduce
$$\overline{f}_t+\overline{\xi\cdot\nabla f}+\overline{\mu(\mathfrak{n})f}=\overline{-\lambda f}+\overline{\lambda\int_{\R^3}T(\xi,\xi')f(t,x,\,\xi')\,d\xi^{'}}.$$
We use $1+|\xi|^m$ to multiply the above equality, then
 \begin{equation*}
 \begin{split}&
\int_{\R^3\times\R^3} \overline{\left(f_t+\xi\cdot\nabla_x f+\mu(\mathfrak{n})f\right)}(1+|\xi|^m)\,dx\,d\xi
\\&=\int_{\R^3\times\R^3}\overline{\left(-\lambda f+\lambda\int_{\R^3}T(\xi,\xi')f(t,x,\,\xi')\,d\xi^{'}\right)}(1+|\xi|^m)\,dx\,d\xi,
 \end{split}
 \end{equation*}
 which in turn gives us
  \begin{equation}
 \label{weak formulation for energy conservation}
 \begin{split}&
\int_{\R^3\times\R^3} \overline{f}(1+|\xi|^m)\,dx\,d\xi+\int_0^t\int_{\R^3\times\R^3} \overline{\mu(\mathfrak{n})f}(1+|\xi|^m)\,dx\,d\xi\,dt
\\&=\int_0^t\int_{\R^3\times\R^3}\overline{\left(-\lambda f+\lambda\int_{\R^3}T(\xi,\xi')f(t,x,\,\xi')\,d\xi^{'}\right)}(1+|\xi|^m)\,dx\,d\xi\,dt
\\&+
\int_{\R^3\times\R^3} \overline{f}_0(1+|\xi|^m)\,dx\,d\xi.
 \end{split}
 \end{equation}

Thanks to Lemma \ref{smooth lemma}, the first term on the right hand side of \eqref{weak formulation for energy conservation} is zero, in particular,
 \begin{equation*}
 \begin{split}
 &\int_0^T\int_{\R^3\times\R^3}\overline{\left(-\lambda f+\lambda\int_{\R^3}T(\xi,\xi')f(t,x,\,\xi')\,d\xi^{'}\right)}(1+|\xi|^m)\,d\xi\,dx\,dt
\\& =-\lambda \int_0^T\int_{\R^3} \overline{\int_{\R^3}f(t,x,\,\xi)(1+|\xi|^m)\,d\xi}\,dx\,dt\\&+\int_0^T\int_{\R^3}\overline{\int_{\R^3}\left(\lambda\int_{\R^3}T(\xi,\xi')f(t,x,\,\xi')(1+|\xi|^m)\,d\xi^{'}\right)\,d\xi}\;\;\,dx\,dt
\\&=0,
 \end{split}
 \end{equation*}
 and hence, \eqref{weak formulation for energy conservation} reduces to the following one
 \begin{equation}
 \label{reduce form for energy conservation}
 \begin{split}&
\int_{\R^3\times\R^3} \overline{f}(1+|\xi|^m)\,dx\,d\xi+\int_0^t\int_{\R^3\times\R^3} \overline{\mu(\mathfrak{n})f}(1+|\xi|^m)\,dx\,d\xi\,dt
\\&=
\int_{\R^3\times\R^3} \overline{f}_0(1+|\xi|^m)\,dx\,d\xi.
 \end{split}
 \end{equation}
\vskip0.3cm

Meanwhile,  Lemma \ref{smooth lemma} gives us
\begin{equation}
\label{smooth m order}
\int_{\R^3\times\R^3} \overline{f}(1+|\xi|^m)\,d\xi\,dx=\int_{\R^3}\overline{\int_{\R^3}f(1+|\xi|^m)\,d\xi}\,dx
\end{equation}
To handle the estimate on $\int_{\R^3}f(1+|\xi|^m)\,d\xi$, we need additional condition and the the following lemma.
 An argument similar to that given above
in the proof to Lemma \ref{Lemma of n-j} shows that,

\begin{Lemma}
\label{Lemma of m order}For any $p\geq 1$, $0\leq t\leq T$, if $f$ is bounded in $L^{\infty}(\R^3\times\R^3\times[0,T]),$ we have
$$\|\int_{\R^3}|\xi|^mf\,d\xi\|_{L^{\infty}(0,T;L^{\frac{3+p}{3+m}}(\R^3))} \leq C_{T}(\|f\|_{L^{\infty}}+1)\left(\int_{\R^3\times\R^3}|\xi|^pf\,d\xi\,dx\right)^{\frac{m+3}{3+p}}.$$
\end{Lemma}
Thus, under the additional condition \eqref{addition condition on m order},  Lemma \ref{Lemma of m order} gives us
\begin{equation*}
\int_{\R^3}(1+|\xi|^m)f\,d\xi
\end{equation*} is uniformly bounded in $L^{\infty}(0,T;L^\alpha(\R^3))$ for some $\alpha>1$.
This, with the help of \eqref{smooth m order},
\begin{equation*}
\int_{\R^3\times\R^3} \overline{f}(1+|\xi|^m)\,dx\,d\xi\to
\int_{\R^3\times\R^3} f(1+|\xi|^m)\,dx\,d\xi,
\end{equation*}
where we adopted the same argument of showing \eqref{limit before R} and \eqref{initial data convergence}.
Similarly, we obtain
\begin{equation*}
\int_{\R^3\times\R^3} \overline{f}_0(1+|\xi|^m)\,dx\,d\xi\to
\int_{\R^3\times\R^3} f_0(1+|\xi|^m)\,dx\,d\xi.
\end{equation*}

Applying Lemma \ref{smooth lemma}, one obtains
\begin{equation}
\label{nonlinear term smooth}
\begin{split}
\int_{\R^3\times\R^3} \overline{\mu(\mathfrak{n})f}(1+|\xi|^m)\,d\xi\,dx
&=\int_{\R^3}\overline{\int_{\R^3}\mu(\mathfrak{n})f(1+|\xi|^m)\,d\xi}\,dx.
\end{split}
\end{equation}
By Lemma \ref{Lemma of n-j} and Lemma \ref{Lemma of m order}, and the additional condition  \eqref{addition condition on m order},  we are able to give a uniform bound on $$\int_{\R^3}\mu(\mathfrak{n})f(1+|\xi|^m)\,d\xi$$
in $L^{\infty}(0,T;L^\alpha(\R^3))$ for some $\alpha\geq 1.$
And hence, using the same argument of showing \eqref{limit before R} and \eqref{initial data convergence} again, as $\varepsilon$ goes to zero,
$$\int_0^t\int_{\R^3\times\R^3} \overline{\mu(\mathfrak{n})f}(1+|\xi|^m)\,dx\,d\xi\,dt\to \int_0^t\int_{\R^3\times\R^3}\mu(\mathfrak{n}) f(1+|\xi|^m)\,dx\,d\xi\,dt.$$

\vskip0.3cm

Finally,
  letting $\varepsilon$ go to zero in \eqref{reduce form for energy conservation},  we obtain
\begin{equation*}
 \begin{split}&
\int_{\R^3\times\R^3}f(1+|\xi|^m)\,dx\,d\xi+\int_0^t\int_{\R^3\times\R^3} \mu(\mathfrak{n})f(1+|\xi|^m)\,dx\,d\xi\,dt
\\&=
\int_{\R^3\times\R^3} f_0(1+|\xi|^m)\,dx\,d\xi.
 \end{split}
 \end{equation*}
\section*{Acknowledgments}

The author thanks Professor Ming Chen at University of Pittsburgh and Professor Alexis Vasseur at the University of Texas at Austin for assistance with their comments and suggestions respectally, that greatly improved the manuscript. The author also thanks Dr. Lei Yao at Northwest University in China for his reading the early version of this manuscript.

\end{document}